\documentclass[leqno,12pt]{amsart}

\usepackage{amsmath,amsthm,amssymb,mathrsfs,mathtools}
\usepackage[all]{xy}

\usepackage{CJKutf8} 
\newcommand{\n}{{\nabla}}

\usepackage[T1]{fontenc}
\usepackage[english]{babel}
\usepackage[margin=1in]{geometry}
\usepackage[onehalfspacing]{setspace}
\usepackage[colorlinks,pagebackref,hypertexnames=false,bookmarks=false]{hyperref}
\usepackage[alphabetic,backrefs]{amsrefs}
\usepackage{xcolor}
\hypersetup{
    allcolors = blue
}
\newcounter{dummy}
\usepackage{enumitem}
\makeatletter
\newcommand\myitem[1][]{\item[#1]\refstepcounter{dummy}\def\@currentlabel{#1}}
\makeatother
\usepackage{ae,aecompl}
\usepackage{enumerate}
\usepackage{multicol}
\usepackage[normalem]{ulem}
\usepackage{tabularx}
\numberwithin{equation}{section}							
\usepackage[nameinlink,noabbrev]{cleveref}
\usepackage{autonum}										

\let\originalleft\left
\let\originalright\right
\renewcommand{\left}{\mathopen{}\mathclose\bgroup\originalleft}
\renewcommand{\right}{\aftergroup\egroup\originalright}

\makeatletter
\renewcommand*{\eqref}[1]{\hyperref[{#1}]{\textup{\tagform@{\ref*{#1}}}}}		
\makeatother

\newcommand\rl{\rm{l}}

\newcommand{\cA}{\mathcal{A}}

\newcommand{\cF}{\mathcal{F}}
\newcommand{\cG}{\mathcal{G}}
\newcommand{\cH}{\mathcal{H}}

\newcommand{\cJ}{\mathcal{J}}
\newcommand{\cK}{\mathcal{K}}

\newcommand{\cP}{\mathcal{P}}

\newcommand{\fg}{{\mathfrak g}}

\renewcommand{\det}{\mathop\mathrm{det}\nolimits}

\newcommand{\Ric}{\rm{Ric}}

\newcommand{\id}{\mathrm{id}}

\renewcommand{\Im}{\mathop{\mathrm{Im}}}

\def\cx{\mathbb{C}}

\def\rl{\mathbb{R}}

\def\Im{\mathrm{Im}}

\def\Ric_g{\mathrm{Ric}}

\def\vol{\mathrm{vol}}

\def\<{\mathopen{}\left<}
\def\>{\right>\mathclose{}}
\def\({\mathopen{}\left(}
\def\){\right)\mathclose{}}

\renewcommand{\epsilon}{\varepsilon}

\renewcommand{\i}{{\sqrt{-1}}}

\newtheorem{theorem}{Theorem}[section]

\newtheorem{lemma}[theorem]{Lemma}
\newtheorem{proposition}[theorem]{Proposition}
\theoremstyle{definition} \newtheorem{definition}[theorem]{Definition}

\newtheorem{remark}[theorem]{Remark}

\crefname{theorem}{Theorem}{Theorems}						
\creflabelformat{theorem}{#2{#1}#3}							
\crefname{Mtheorem}{Main Theorem}{Main Theorems}			
\creflabelformat{Mtheorem}{#2{#1}#3}						
\crefname{lemma}{Lemma}{Lemmata}							
\creflabelformat{lemma}{#2{#1}#3}							
\crefname{corollary}{Corollary}{Corollaries}				
\creflabelformat{corollary}{#2{#1}#3}						
\crefname{proposition}{Proposition}{Propositions}			
\creflabelformat{proposition}{#2{#1}#3}						
\crefname{ineq}{inequality}{inequalities}					
\creflabelformat{ineq}{#2{\upshape(#1)}#3}					
\crefname{cond}{condition}{conditions}						
\creflabelformat{cond}{#2{\upshape(#1)}#3}					
\crefname{hypoth}{Hypothesis}{Hypotheses}					
\creflabelformat{hypoth}{#2{#1}#3}							
\crefname{def}{Definition}{Definitions}						
\creflabelformat{def}{#2{#1}#3}								
\crefname{appsec}{Appendix}{Appendices}

\crefname{sec}{Section}{Sections}

\begin{document}

\author{Kotaro Kawai}
\address{Yanqi Lake Beijing Institute of Mathematical Sciences and Applications, 
No. 544, Hefangkou Village, Huaibei Town, Huairou District, Beijing, 101408, 
China}
\address{
Department of Mathematics, Osaka Metropolitan University, 3-3-138, Sugimoto, Sumiyoshi-ku, Osaka, 558-8585, Japan}
\email{\href{mailto:kkawai@bimsa.cn}{kkawai@bimsa.cn}}


\title{Some observations on deformed Donaldson--Thomas connections}
\thanks{
This work was supported by JSPS KAKENHI Grant Number JP21K03231 
and 
MEXT Promotion of Distinctive Joint Research Center Program JPMXP0723833165. 
}

\maketitle

\begin{abstract}
A deformed Donaldson--Thomas (dDT) connection is a Hermitian connection of a Hermitian line bundle over a $G_2$-manifold $X$ satisfying 
a certain nonlinear PDE. This is considered to be the mirror of a (co)associative cycle in the context of mirror symmetry. 
It can also be considered as an analogue of a $G_2$-instanton. 
In this paper, we see that 
some important observations that appear in other geometric problems are also found 
in the dDT case as follows. 

(1) 
A dDT connection exists 
if a 7-manifold has full holonomy $G_2$ and the $G_2$-structure is ``sufficiently large". 
(2)
The dDT equation is described as the zero of a certain multi-moment map. 
(3) 
The gradient flow equation of a Chern-Simons type functional of Karigiannis and Leung, 
whose critical points are dDT connections, 
agrees with the ${\rm Spin}(7)$ version of the dDT equation on a cylinder 
with respect to a certain metric on a certain space. 
This can be considered as an analogue of the observation in instanton Floer homology for 3-manifolds. 
\end{abstract}

\setcounter{tocdepth}{2}
\tableofcontents

\section{Introduction}
Let $X^7$ be a $7$-manifold with a $\rm G_2$-structure $\varphi\in\Omega^3(X)$. 
For the definition of $G_2$-structures, 
see for example \cite[Section 2.2]{kawai2021mirror}. 
We use the same sign convention there. 
Denote by $g$, $\vol$ and $\ast$ the induced Riemannian metric, 
volume form and Hodge star operator, respectively. 
Let $(L,h)\to X$ be a smooth complex Hermitian line bundle over $X$. 
We denote by $\mathcal{A}_0$ the affine space of Hermitian connections on $(L,h)$. 
Given $\nabla\in\mathcal{A}_0$, we regard its curvature $F_{\nabla}$ as a $\sqrt{-1}\mathbb{R}$-valued closed $2$-form on $X$. 

\begin{definition} \label{def:dDT}
A Hermitian connection $\nabla\in\mathcal{A}_0$ satisfying
\begin{equation}\label{eq: dDT}
    \frac{1}{6}F_{\nabla}^3 + F_{\nabla}\wedge\ast\varphi = 0 
\end{equation} is called a \textbf{deformed Donaldson-Thomas (dDT) connection}.
\end{definition}

DDT connections appeared in the context of mirror symmetry. 
They were introduced in \cite{lee2009geometric} as ``mirrors" of calibrated (associative) submanifolds. 
Historically, 
deformed Hermitian Yang--Mills (dHYM) connections
were introduced first in \cite{lyz2000FM} as ``mirrors" of special Lagrangian submanifolds. 
There is also a similar notion of dDT connections for a manifold with a ${\rm Spin}(7)$-structure 
(\cites{lee2009geometric, kawai2021FM}). 
As the names indicate, 
dDT connections can also be considered as 
analogues of Donaldson--Thomas connections ($G_2$-instantons).

Thus it is natural to expect that dDT connections would have similar properties to 
associative submanifolds and $G_2$-instantons. 
We show that it is indeed the case in \cites{kawai2020deformation, kawai2021mirror}. 
For example, the moduli space of dDT connections 
is $b^1$-dimensional and canonically orientable if we perturb the $G_2$-structure. 
Any dDT connection on a compact $G_2$-manifold is a global minimizer of the ``mirror volume" 
and its value is topological by the ``mirror'' of associator equality.  
We could also prove similar statements in the ${\rm Spin}(7)$ case 
in \cites{kawai2021deformationSpin(7), kawai2021mirror}. 
Moreover, 
dDT connections are given by critical points of the Chern-Simons type functional 
in \cite[Theorem 5.13]{karigiannis2009hodge}. 
The variational characterization is known only for the $G_2$ case, 
and no such characterization is known for the ${\rm Spin}(7)$ case. \\

This paper is organized as follows. 
In Section \ref{sec:lrl}, we study the existence of a dDT connection. 
Known examples of dDT connections are either trivial or 
constructed in \cites{lotay2020examples, fowdar2022examples}, 
and are very few in number. 
So it would be important to consider the existence problem. 
We first see that the formal ``large radius limit" of the defining equation of dDT connections 
is that of $G_2$-instantons. 
Thus it is natural to expect that dDT connections for a ``sufficiently large" $G_2$-structure 
will behave like $G_2$-instantons. 
Moreover, it is known that 
any complex Hermitian line bundle admits a $G_2$-instanton on a compact holonomy $G_2$-manifold. 
Then we show the following from these facts. 

\begin{theorem}[Theorem \ref{thm:exist dDT}] 
Suppose that $(X, \varphi)$ is a compact holonomy $G_2$-manifold.
Let $(L,h) \to X$ be a smooth complex Hermitian line bundle over $X$. 
If the $G_2$-structure is ``sufficiently large", there exists a dDT connection. 
\end{theorem}

In Section \ref{sec:mmm}, we formulate the dDT equation in terms of a multi-moment map. 
The multi-moment map is a generalization of the moment map 
introduced in \cites{Madsen2012multi,Madsen2013closed}. 
The dHYM equation is described as the zero of a certain moment map 
on a infinite dimensional symplectic manifold 
(\cite[Section 2]{collins2021moment}, \cite[Section 2.1]{collins2021survey}). 
Analogously, we show that the dDT equation is described as the zero of a certain multi-moment map. 

\begin{theorem}[Theorem \ref{thm:mmm}]
The dDT equation is described as the zero of a certain multi-moment map. 
\end{theorem}

In Section \ref{sec:grad}, 
we study the gradient flow of the Karigiannis-Leung functional 
introduced in \cite{karigiannis2009hodge} whose critical points are dDT connections. 
It is known that 
the gradient flow equation of the Chern-Simons functional
on an oriented 3-manifold $X^3$ agrees with the ASD equation on $\rl \times X^3$. 
This is an important observation in instanton Floer homology for 3-manifolds. 
We show that there is an analogous relation between 
dDT equations for $G_2$- and ${\rm Spin}(7)$-manifolds 
using the Karigiannis-Leung functional. 
This will establish a new link between 
3, 4-manifold theory and $G_2$-, ${\rm Spin}(7)$-geometry, 
and we might define analogues of instanton Floer homology using dDT connections.

\begin{theorem}[Theorem \ref{thm:gradient flow}]
The gradient flow equation of a Chern-Simons type functional of Karigiannis and Leung, 
whose critical points are dDT connections, 
agrees with the ${\rm Spin}(7)$ version of the dDT equation on a cylinder 
with respect to a certain metric on a certain space. 
\end{theorem}

\section{Large radius limit} \label{sec:lrl}

In this section, we show the existence of a dDT connection 
if a 7-manifold has full holonomy $G_2$ and the $G_2$-structure is ``sufficiently large".

Suppose that $(X, \varphi)$ is a compact holonomy $G_2$-manifold.
Let $(L,h) \to X$ be a smooth complex Hermitian line bundle over $X$. 
Set 
\[
\begin{aligned}
\mathcal{A}_{0}=\{\, \mbox{Hermitian connections of }(L,h) \,\}
= \nabla_0 + \i \Omega^1 \cdot \id_L, 
\end{aligned}
 \]
where $\nabla_0 \in \cA_{0}$ is any fixed connection 
and $\Omega^1$ is the space of 1-forms on $X$. 
Denote by $\cG_U$ the group of unitary gauge transformations of $(L,h)$, 
which acts on $\cA_{0}$. 
Explicitly, 
\[
\cG_U= \{\, f \cdot \id_L \mid f \in \Omega^0_{\cx}, \ |f|=1 \,\} \cong C^\infty(X, S^1), 
\]
where $\Omega^0_{\cx}$ is the space of $\cx$-valued smooth functions, 
and the action $\cG_U \times \cA_0 \rightarrow \cA_0$ is defined by 
$(\lambda, \nabla) \mapsto 
\lambda^* \nabla:= \lambda^{-1} \circ \nabla \circ \lambda$.
When $\lambda=f \cdot \id_L$ for $f \in C^\infty(X, S^1)$, we have 
\begin{align} \label{eq:Gu action} 
\lambda^* \nabla 
= \lambda^{-1} \circ \nabla \circ \lambda = \nabla + f^{-1}df \cdot \id_L.     
\end{align}
Thus the $\cG_U$-orbit through $\n \in \cA_0$ is given by $\n + \cK_U \cdot \id_L$, where 
\begin{equation} \label{eq:Gu orbit G2}
\cK_U:= \left\{ f^{-1} d f \in \i \Omega^1 \; \middle| \; f \in \Omega^{0}_{\cx}, \ |f|=1 \right\}. 
\end{equation}
Note that the curvature 2-form $F_\nabla$ is invariant under the action of $\cG_U$. \\

Consider the family of $G_2$-structures 
$$\{ \varphi_r:=r^3 \varphi \}_{r>0},$$ 
all of which induce holonomy $G_2$ metrics. 
The defining equation of dDT connections with respect to $\varphi_r$ is given by 
$$
0=\cF_r(\n) :=\frac{1}{6} F_\n^3 + r^4 F_\n \wedge * \varphi.
$$
Thus, formally taking the "large radius limit", 
which means the leading behaviour of $\cF_r(\n)$ as $r \to \infty$, we obtain 
$$
F_\n \wedge * \varphi =0. 
$$
This is exactly the defining equation of $G_2$-instantons. 
Thus it is natural to expect that dDT connections for a sufficiently large $G_2$-structure 
will behave like $G_2$-instantons. 
The following would be well-known for $G_2$-instantons on a smooth complex Hermitian line bundle, but we give the proof for completeness.

\begin{lemma} \label{lem:ex G2inst} 
On a compact holonomy $G_2$-manifold $(X^7,\varphi)$, 
there is a unique $G_2$-instanton 
on a smooth complex Hermitian line bundle $L\to X$ 
up to the action of $\cG_U$. 
\end{lemma}

\begin{proof}
For any $\n \in \cA_0$, we have $d F_\n=0$. 
So it defines a cohomology class 
$[F_\n] \in \i H^2(X, \rl)$, which is known to be equal to 
$-2 \pi \i c_1(L)$. 
Then there exists a 1-form $\alpha \in \i \Omega^1$ such that 
$F_\n + d \alpha$ is harmonic by Hodge theory.

Denote by $\Omega^k_\ell \subset \Omega^k$ 
the subspace of the space of $k$-forms corresponding to 
the $\ell$-dimensional irreducible representation of $G_2$. 
For more details, see for example \cite[Section 2.2]{kawai2021mirror}. 
Denote by $\cH^k$ the space of harmonic $k$-forms on $X$ 
and set $\cH^k_\ell= \cH^k \cap \Omega^k_\ell$. 
Then by \cite[Theorem 10.2.4]{joyce2000compact}, 
we have 
$\cH^2_7 \cong \cH^1_7 = \cH^1 = \{ 0 \}$. 
Thus we have 
$$
F_{\n + \alpha \cdot \id_L} = F_\n + d \alpha 
\in \i \cH^2 = \i \cH^2_7 \oplus \cH^2_{14} = \i \cH^2_{14}, 
$$
which implies that 
$F_{\n + \alpha \cdot \id_L} \wedge * \varphi =0$. 

If $\n'=\n + (\alpha + \alpha') \cdot \id_L$ for $\alpha' \in \i \Omega^1$ 
is also a $G_2$-instanton, we have 
$0=F_{\n'} \wedge * \varphi = d \alpha' \wedge * \varphi$, 
which is equivalent to 
\begin{align} \label{eq:ex G2inst}
-d \alpha' = * (d \alpha' \wedge \varphi) = * d (\alpha' \wedge \varphi). 
\end{align}
Since $d \Omega^1 \cap d^* \Omega^3 = \{ 0 \}$, we have $d \alpha'=0$. 
Since 
$H^1(X,\rl)= \{ 0 \}$ 
by \cite[Theorem 10.2.4]{joyce2000compact} again 
and 
$\i \rl$-valued exact 1-forms are contained in $\cK_U$, 
the $G_2$-instanton is unique up to the action of $\cG_U$. 
\end{proof}

Using this, we can show the following.

\begin{theorem} \label{thm:exist dDT}
Suppose that $(X, \varphi)$ is a compact holonomy $G_2$-manifold.
Let $(L,h) \to X$ be a smooth complex Hermitian line bundle over $X$. 
Then for sufficiently large $r>0$, there exists a dDT connection with respect to $\varphi_r$. 
\end{theorem}

\begin{proof}
Define a map $\cF :[0,1] \times \cA_0 \to \i d \Omega^5$ by 
$$
\cF (s,\n) = \frac{s^4}{6} F_\n^3+F_\n \wedge * \varphi. 
$$
Then $\cF (0, \cdot)^{-1} (0)/\cG_U$, which is a point by Lemma \ref{lem:ex G2inst}, 
is the moduli space of $G_2$-instantons with respect to $\varphi$ 
and 
$\cF (s, \cdot)^{-1} (0)/\cG_U$ for $s \neq 0$  
is the moduli space of dDT connections with respect to $\varphi_{1/s}$. 

We want to apply the implicit function theorem to show the statement. 
Fix a $G_2$-instanton $\n_0 \in \cF (0, \cdot)^{-1} (0)$, 
whose existence is guaranteed by Lemma \ref{lem:ex G2inst}.
Denote by the linearization $(d \cF)_{(0,\n_0)}: \rl \oplus \i \Omega^1 \to \i d \Omega^5$ 
of $\cF$ at $(0,\n_0)$. Then we have  
$$
(d \cF)_{(0,\n_0)} (0, \i b) = \i db \wedge * \varphi. 
$$

\begin{lemma} \label{lem:ex ker im}
We have 
$$
\ker (d \cF)_{(0,\n_0)} = \rl \oplus \i d \Omega^0, \qquad \Im{(d \cF)_{(0,\n_0)}} = \i d \Omega^5.
$$
\end{lemma}
\begin{proof}
The first equation is proved as in \eqref{eq:ex G2inst}. 
For the second equation, 
the Hodge decomposition implies that 
$d^* \Omega^2=d^* d \Omega^1$. For any $b \in \Omega^1$, we have 
$$
d^* d b = d^* (d b + * (\varphi \wedge d b)) \in d^* \Omega^2_7,  
$$
where we use the fact that $\varphi$ is closed. This implies that $d^* \Omega^2=d^* \Omega^2_7$. 
Then 
$$d \Omega^5=*d^* \Omega^2=*d^* \Omega^2_7=d \Omega^5_7.$$ 
Since $\Omega^5_7$ is spanned by $b \wedge * \varphi$ for $b \in \Omega^1$, 
the proof is completed. 
\end{proof}

By the Hodge decomposition and $H^1(X,\rl)= \{ 0 \}$, 
we have $\Omega^1=d \Omega^0 \oplus d^* \Omega^2$. 
By this and Lemma \ref{lem:ex ker im}, we see that 
$(d \cF)_{(0,\n_0)}|_{\i d^* \Omega^2}: \i d^* \Omega^2 \to \i d \Omega^5$ is an isomorphism. 
Hence, we can apply the implicit function theorem (after the Banach completion)
and we see that $\cF (s, \cdot)^{-1} (0) \neq \emptyset$ for sufficiently small $s$.

Finally, we explain how to recover the regularity of elements in
$\cF (s, \cdot)^{-1} (0)$ after the Banach completion. 
Since the curvature is invariant under the addition of closed 1-forms, 
there exists $a_s \in \Omega^1$ such that 
\begin{align} 
\cF (s, \n_0 + \i a_s \cdot \id_L)=0, \qquad d^* a_s=0 \tag{$*_s$}
\end{align}
for sufficiently small $s$. 
In particular, 
$(*_0)$ is given by $d a_0 \wedge * \varphi=d^* a_0=0$, 
which is an overdetermined elliptic equation. 
To be overdetermined elliptic is an open condition, so we see 
that $(*_s)$ is also overdetermined elliptic for sufficiently small $s$. 
Hence we can find a smooth element in $\cF (s, \cdot)^{-1} (0)$ around $(0,\n_0)$ and the proof is completed. 
\end{proof}

\section{The multi-moment map} \label{sec:mmm}
It is known that there is a moment map picture in the dHYM case. 
In particular, the dHYM equation is described as the zero of a certain moment map 
on a infinite dimensional symplectic manifold. 
See for example \cite[Section 2]{collins2021moment} or 
the survey article \cite[Section 2.1]{collins2021survey}. 
Analogously, we show that the dDT equation is described as the zero of a certain multi-moment map. 
First, recall the definition of the multi-moment map 
in \cites{Madsen2012multi,Madsen2013closed}. 

\begin{definition} \label{def:mmm}
Let $X$ be a smooth manifold 
and $c \in \Omega^3$ be a closed 3-form on $X$. 
Suppose that a Lie group $G$ acts on $X$ preserving $c$. 
Denote by $\fg$ the Lie algebra of $G$ and 
set 
$$
\cP_\fg=\ker (L: \Lambda^2 \fg \to \fg) \subset \Lambda^2 \fg, 
$$ 
where $L$ is the linear map induced by the Lie bracket. 
(Note that $\cP_\fg = \Lambda^2 \fg$ if $G$ is abelian.) 
Denote by $u^*$ the vector field on $X$ generated by $u \in \fg$.  
For a two vector $p = \sum_j u_j \wedge v_j \in \Lambda^2 \fg$, 
set 
$$
p^* = \sum_j u_j^* \wedge v_j^*, \qquad 
i(p^*) c = \sum_j c(u_j^*, v_j^*, \cdot). 
$$
Denote by $\langle \cdot, \cdot \rangle: \Lambda^2 \fg^* \times \Lambda^2 \fg \to \rl$ the canonical pairing. 

Then a map $\nu:M \to \cP_\fg^*$ is called a {\bf multi-moment map} if 
it is $G$-equivariant and satisfies 
$$
d \langle \nu, p \rangle = i(p^*) c
$$
for any $p \in \cP_\fg$. 
\end{definition}

Let $X$ be a compact 7-manifold with a coclosed $G_2$-structure $\varphi$ 
($d * \varphi=0$) and 
$(L,h) \to X$ be a smooth complex Hermitian line bundle over $X$. 
Let $\cA_0$ be the space of Hermitian connections of $(L,h)$. 
Define a map $\cF_{G_2}: \cA_0 \to \i \Omega^6$ by 
$$
\cF_{G_2} (\n)= \frac{1}{6} F_\n^3 + F_\n \wedge * \varphi. 
$$
Then the space of dDT connections is given by $\cF_{G_2}^{-1}(0)$. 
Denote by $\cG_U$ the group of unitary gauge transformations of $(L,h)$ 
acting $\cA_0$ canonically as in \eqref{eq:Gu action}. 
Since $\cG_U = C^\infty(X,S^1)$, the Lie algebra $\fg_U$ of $\cG_U$ is identified with 
the space $\i \Omega^0$ of $\i \rl$-valued functions on $X$. 
Note that $\cP_\fg=\Lambda^2 \fg$ since $\cG_U$ is abelian. 
Define a 3-form $\Theta \in \Omega^3(\cA_0)$ on $\cA_0$ by 
$$
\Theta_\n (\alpha_1,\alpha_2,\alpha_3) = 
\i \int_X \alpha_1 \wedge \alpha_2 \wedge \alpha_3 
\wedge \left ( \frac{1}{2} F_\n^2 + * \varphi \right),  
$$
where $\n \in \cA_0$ and $\alpha_1, \alpha_2, \alpha_3 \in \i \Omega^1 = T_\n \cA_0$. 
We first show the following as required in Definition \ref{def:mmm}.

\begin{lemma} \label{lem:3-form closed}
The 3-form $\Theta$ is $\cG_U$-invariant and closed. 
\end{lemma}
\begin{proof}
Take any $\n \in \cA_0$, $\lambda = f \cdot \id_L \in \cG_U$, 
where $f \in C^\infty (X, S^1)$, and 
$\alpha_1, \alpha_2, \alpha_3, \alpha_4 \in \i \Omega^1 \cong T_\n \cA_0$. 
Identify $\alpha_j$ with a vector field on $\cA_0$ by 
$$
(\alpha_j)_{\widetilde \n} 
= \left. \frac{d}{dt} \left( \widetilde \n + t \alpha_j \cdot \id_L \right) 
\right|_{t=0}  
$$
for $\widetilde \n \in \cA_0$. 
We first show the $\cG_U$-invariance of $\Theta$. That is, 
\begin{align} \label{eq:Gu inv Theta}
\Theta_{\lambda^* \n} 
\left(\lambda_* (\alpha_1), \lambda_* (\alpha_2), \lambda_* (\alpha_3) \right)
= 
\Theta_{\n} 
\left( \alpha_1, \alpha_2, \alpha_3 \right).     
\end{align}
By \eqref{eq:Gu action}, we compute 
\begin{align}
\lambda_* (\alpha_j)_\n
= 
\lambda_* 
\left. \frac{d}{dt} \left(\n + t \alpha_j \cdot \id_L \right) \right|_{t=0}  
= 
\left. \frac{d}{dt} \left(\n + (t \alpha_j+ f^{-1} df) \cdot \id_L 
\right) \right|_{t=0}  
= 
(\alpha_j)_{\lambda^* \n}. 
\end{align}
Since $F_{\lambda^* \n}=F_\n$, we obtain \eqref{eq:Gu inv Theta}. 

Next, we show the closedness of $\Theta$. 
Note that $[\alpha_i, \alpha_j]=0$. Then it follows that  
\begin{align}
d \Theta (\alpha_1, \alpha_2, \alpha_3, \alpha_4)
=& 
\alpha_1 \left(\Theta (\alpha_2, \alpha_3, \alpha_4)\right)
- \alpha_2 \left(\Theta (\alpha_1, \alpha_3, \alpha_4)\right) \\
&+ \alpha_3 \left(\Theta (\alpha_1, \alpha_2, \alpha_4)\right)
- \alpha_4 \left(\Theta (\alpha_1, \alpha_2, \alpha_3)\right). 
\end{align}
Since 
\begin{align}
\alpha_i \left(\Theta (\alpha_j, \alpha_k, \alpha_\ell)\right)_\n 
=& 
\i 
\left. \frac{d}{dt} 
\int_X \alpha_j \wedge \alpha_k \wedge \alpha_\ell \wedge 
\left( \frac{1}{2} F_{\n+t \alpha_i \cdot \id_L}^2 + * \varphi 
\right) 
\right|_{t=0} \\
=&
\i \int_X \alpha_j \wedge \alpha_k \wedge \alpha_\ell \wedge d \alpha_i \wedge F_\n, 
\end{align}
we have 
$$
(d \Theta)_\n (\alpha_1, \alpha_2, \alpha_3, \alpha_4)
= 
\i \int_X d( \alpha_1 \wedge \alpha_2 \wedge \alpha_3 \wedge \alpha_4 \wedge F_\n) =0, 
$$
which implies that $d \Theta =0$. 
\end{proof}

We also need the following lemma.

\begin{lemma} \label{lem:omega1 fn}
We have 
$$
\Omega^1 = \left\{ \sum_{j=1}^N f^j_1 d f^j_2 \; \middle| \; N \in \mathbb{N}, f^j_1, f^j_2 \in \Omega^0 \right \}. 
$$
\end{lemma}

\begin{proof}
Take any 1-form $\alpha \in \Omega^1$. We first show that for any $x\in X$, 
there exists an open neighborhood $U_x$ of $x$ 
and smooth functions 
$\{ \widetilde f_{x, j}^{1}, \widetilde f_{x, j}^{2} \}_{j=1}^{7}$ on $X$ such that 
\begin{align} \label{eq:alpha local}
\alpha |_{U_x}= \sum ^{7}_{j=1} \widetilde f_{x,j}^{1} \ d \widetilde f_{x,j}^{2}|_{U_x}. 
\end{align}
Indeed, take any local coordinates $(V, (x^1, \cdots, x^7))$ of $x$ and set 
$$
\alpha|_{V}=\sum^{7}_{j=1}\alpha _{j}dx^{j}. 
$$
We can take a cutoff function $h$ such that 
$h$ has compact support in $V$ and $h=1$ on an open neighborhood $U_x$ of $x$. 
Then setting 
$\widetilde f_{x,j}^1= h \alpha_{j}$ and $\widetilde f_{x,j}^2= h x_j$, 
which are smooth functions on $X$, we obtain \eqref{eq:alpha local}. 

Since $\{ U_x \}_{x \in X}$ is an open cover of $X$ and $X$ is compact, 
there exists $x_1, \cdots, x_N \in X$ such that $\{ U_{x_p} \}_{p=1}^N$ covers $X$. 
Denote by $\{ h_p \}_{p=1}^N$ the partition of unity subordinate to $\{ U_{x_p} \}_{p=1}^N$. 
Set 
$$
f_{p,j}^{1}= h_{p} \widetilde {f}_{x_p, j}^1 \qquad 
f_{p,j}^{2}= \widetilde f_{x_p, j}^{2}. 
$$
Then we have $\alpha =\sum ^{N}_{p=1}\sum ^{7}_{j=1}f_{p,j}^{1}df_{p,j}^{2}$. 
Indeed, take any $x \in X$. 
We may assume that 
$x \in U_{x_1} \cap \cdots \cap U_{x_k}$ 
and $x \not\in U_p$ for $p=k+1, \cdots, N$. 
Then 
$\sum ^{7}_{j=1} \left(\widetilde f_{x_p,j}^{1} \ d \widetilde f_{x_p, j}^{2} 
\right)_x = \alpha_x$ 
for $p=1, \cdots, k$ by \eqref{eq:alpha local} and 
$h_p (x)=0$ for $p=k+1, \cdots, N$. Hence 
$$
\sum ^{N}_{p=1}\sum ^{7}_{j=1} \left(f_{p,j}^{1}df_{p,j}^{2} \right)_x
= 
\sum ^{k}_{p=1} h_p (x) \sum ^{7}_{j=1} 
\left( \widetilde f_{x_p,j}^{1} \ d \widetilde f_{x_p, j}^{2} \right)_x
= 
\sum ^{k}_{p=1} h_p (x) \alpha_x 
= 
\sum ^{N}_{p=1} h_p (x) \alpha_x
=
\alpha_x. 
$$
\end{proof}

Denote by $Z^6$ the space of closed 6-forms on $X$. 
Define a map $\iota_{Z^6}: Z^6 \to \Lambda^2 \fg_U^*$ by 
$$
\iota_{Z^6}(\xi) (f_1,f_2)
= \int_X \xi \wedge \frac{1}{2} (f_1 d f_2 -f_2 d f_1)
= \int_X \xi \wedge f_1 d f_2
$$
for $\xi \in Z^6$ and $f_1,f_2 \in \i \Omega^0 = \fg_U$.

\begin{theorem} \label{thm:mmm}
Define a $\cG_U$-invariant map $\nu: \cA_0 \to \Lambda^2 \fg_U^*$ by 
$$
\nu (\n)= \iota_{Z^6} (\i \cF_{G_2}(\n)). 
$$
Then we have 
$d \langle \nu, p \rangle = i(p^*) \Theta$
for any $p \in \Lambda^2 \fg_U$. 
\end{theorem}

Since we assume that $d * \varphi =0$, 
we see that $\i \cF_{G_2}(\n) \in Z^6$ for any $\n \in \cA_0$. 
By Lemma \ref{lem:omega1 fn}, $\iota_{Z^6}$ is injective. 
Hence we have $\nu^{-1} (0)= \cF_{G_2}^{-1} (0)$. 
In this sense, we can regard the dDT equation as the zero of a multi-moment map. 

\begin{proof}
First note that the vector field $f^*$ generated by $f \in \i \Omega^0 = \fg_U$ 
is given by 
$$
f^*_\n 
= \left. \frac{d}{dt} (e^{t f})^* \n \right|_{t=0} 
= \left. \frac{d}{dt} \left( \n + e^{-t f} d e^{t f} \cdot \id_L \right) \right|_{t=0} 
= df 
$$
at $\n \in \cA_0$. 
Hence for any $f_1, f_2 \in \i \Omega^0 = \fg_U$ and $\alpha \in \i \Omega^1 = T_\n \cA_0$, 
we have 
\begin{align}
&\Theta_\n ((f_1^*)_\n, (f_2^*)_\n, \alpha) \\
=&
\i \int_X d f_1 \wedge d f_2 \wedge \alpha \wedge \left ( \frac{1}{2} F_\n^2 + * \varphi \right)\\
=&
\i \int_X f_1 d f_2 \wedge d\alpha \wedge \left ( \frac{1}{2} F_\n^2 + * \varphi \right)
=
\i \int_X f_1 d f_2 \wedge (d \cF_{G_2})_\n (\alpha), 
\end{align}
where $(d \cF_{G_2})_\n: \i \Omega^1 \to \i \Omega^6$ is the linearization of  
$\cF_{G_2}$ at $\n \in \cA_0$. 
Hence we obtain 
\begin{align}
\Theta_\n ((f_1^*)_\n, (f_2^*)_\n, \alpha) 
=&
\iota_{Z^6} (\i (d \cF_{G_2})_\n (\alpha)) (f_1, f_2)\\
=&
\left. \frac{d}{d t} \iota_{Z^6} (\i \cF_{G_2} (\n+t \alpha \cdot \id_L)) \right|_{t=0} (f_1, f_2)
= 
(d \langle \nu, f_1 \wedge f_2 \rangle)_\n (\alpha). 
\end{align}
\end{proof}

In the dHYM case, the ``$\cJ$ functional" defined 
in \cite[Remark 2.15]{collins2021moment} or \cite[Lemma 2.6 (ii)]{collins2021survey} 
is convex along geodesics and the critical points are solutions of the dHYM equation. 
Hence it plays an important role in the existence problem. 

In the dDT case, there is a functional whose critical points are dDT connections. 
See Section \ref{sec:KL func}. 
However, no metric has yet been found that makes the functional convex along geodesics. 
Since no such results have been found for associative submanifolds, 
it might be difficult to relate the functional to the existence problem. 

However, as we see in the next section, 
we have an observation as in the case of instanton Floer homology for 3-manifolds 
by using the functional in Section \ref{sec:KL func}. 
We might develop the theory like instanton Floer homology 
using dDT connections.

\section{Gradient flow of the Karigiannis-Leung functional} \label{sec:grad}

It is known that 
the gradient flow equation of the Chern-Simons functional
on an oriented 3-manifold $X^3$ agrees with the ASD equation on $\rl \times X^3$. 
See for example \cite[Section 2.5.3]{Donaldson2002}. 
This is an important observation in instanton Floer homology for 3-manifolds. 
We show that there is an analogous relation between 
dDT equations for $G_2$- and ${\rm Spin}(7)$-manifolds.

Let $X^7$ be a 7-manifold with a $G_2$-structure $\varphi$ and 
$(L,h) \to X^7$ be a smooth complex Hermitian line bundle over $X^7$. 
Let $\{ \n_t \}_{t \in \rl}$ be a family of Hermitian connections of $(L,h) \to X^7$. 
We identify this with a connection $\widetilde \n$ of 
$\pi^*L \to \rl \times X^7$, where $\pi: \rl \times X^7 \to X^7$ is the projection. 
If we set 
$$
\n_t=\n_0+ \i a_t \cdot \id_L, 
$$
where $a_t \in \Omega^1(X^7)$, 
we have 
$\widetilde \n = \pi^* \n_0 + \i \pi^* a_t \cdot \id_{\pi^* L}$
and 
the curvature $F_{\widetilde \n}$ of $\widetilde \n$ is given by 
$$
F_{\widetilde \n} = \sqrt{-1} dt \wedge \frac{\partial \pi^* a_t}{\partial t} + \pi^* F_{\n_t}.  
$$

\subsection{The ${\rm Spin}(7)$-dDT condition on $\rl \times X^7$} \label{sec:Spin7 cyl}
The product $\rl \times X^7$ admits a canonical ${\rm Spin}(7)$-structure. 
We write down the condition that 
$F_{\widetilde \n}$ is a ${\rm Spin}(7)$-dDT connection, 
a dDT connection for a manifold with a ${\rm Spin}(7)$-structure. 
For simplicity, set 
$$
\dot a_t:=\frac{\partial \pi^* a_t}{\partial t}, \qquad E_t:= - \sqrt{-1} \pi^* F_{\n_t}. 
$$

\begin{lemma}
The connection $\widetilde \n$ is a ${\rm Spin}(7)$-dDT connection if and only if 
\begin{align}
-* \varphi \wedge E_t + \frac{1}{6} E_t^3
- \left( 1-\frac{1}{2} *(\varphi \wedge E_t^2) \right) * \dot a_t 
+ * (\dot a_t \wedge E_t \wedge \varphi) \wedge * E_t &=0 \label{eq:Spin7cyl 1}\\
\frac{1}{2} \varphi \wedge * E_t^2 - \dot a_t \wedge E_t \wedge \varphi &=0. \label{eq:Spin7cyl 2}
\end{align}
\end{lemma}

\begin{proof}
Denote by $*_8$ and $*=*_7$ the Hodge star operators on $\rl \times X^7$ and $X^7$, respectively. 
Then, $\widetilde \n$ is a ${\rm Spin}(7)$-dDT connection 
(in the sense of \cite[Definition 1.3]{kawai2021FM})
if and only if 
\begin{align}\label{eq:Spin7cyl 3}
\left\langle F_{\widetilde \n} + \frac{1}{6} *_8 F_{\widetilde \n}^3, \ dt \wedge b + i(b^\sharp) \varphi \right\rangle=0,\qquad
\left\langle F_{\widetilde \n}^2, \ dt \wedge i(b^\sharp) * \varphi - b \wedge \varphi \right\rangle=0
\end{align}
for any $b \in \Omega^1(X^7)$ by \cite[Lemma 3.4]{kawai2021FM}. 
Since 
$$
\frac{1}{6} *_8 F_{\widetilde \n}^3 
= - \frac{\sqrt{-1}}{6} *_8 (3 dt \wedge \dot a_t \wedge E_t^2 +E_t^3)
= \sqrt{-1} \left( -\frac{1}{2} * \left( \dot a_t \wedge E_t^2 \right) - \frac{1}{6} dt \wedge * E_t^3 \right), 
$$
\eqref{eq:Spin7cyl 3} is equivalent to 
\begin{align}
\left\langle \dot a_t - \frac{1}{6} * E_t^3, b \right\rangle
+ \left\langle E_t - \frac{1}{2} * \left( \dot a_t \wedge E_t^2 \right), i(b^\sharp) \varphi \right \rangle =0, 
\label{eq:Spin7cyl 1 0}\\
\left\langle 2 \dot a_t \wedge E_t, i(b^\sharp) * \varphi \right\rangle
- \left\langle E_t^2, b \wedge \varphi \right \rangle =0.  \label{eq:Spin7cyl 2 0}
\end{align}
We compute 
$$
\left\langle E_t, i(b^\sharp) \varphi \right \rangle 
= * (E_t \wedge * (i(b^\sharp) \varphi))
= * (E_t \wedge b \wedge * \varphi) = \langle * \varphi \wedge E_t, * b \rangle
$$
and 
\begin{align}
\left\langle - \frac{1}{2} * \left( \dot a_t \wedge E_t^2 \right), i(b^\sharp) \varphi \right \rangle 
=&
- \frac{1}{2} * \left( \dot a_t \wedge E_t^2 \wedge i(b^\sharp) \varphi \right) \\
=&
- \frac{1}{2} * \left( i(b^\sharp) (\dot a_t \wedge E_t^2) \wedge  \varphi \right) \\
=&
- \frac{1}{2} * \left( E_t^2 \wedge  \varphi \right) \cdot \langle \dot a_t, b \rangle 
+ 
* \left( \dot a_t \wedge E_t \wedge (i(b^\sharp) E_t)  \wedge  \varphi \right). 
\end{align}
Since 
$i(b^\sharp) E_t = -* (b \wedge * E_t)$, we have 
\begin{align}
* \left( \dot a_t \wedge E_t \wedge (i(b^\sharp) E_t)  \wedge  \varphi \right)
=
\langle \dot a_t \wedge E_t \wedge \varphi, b \wedge * E_t \rangle
=
- \langle * (\dot a_t \wedge E_t \wedge \varphi) \wedge * E_t, * b \rangle. 
\end{align}
Then, we see that \eqref{eq:Spin7cyl 1 0} is equivalent to \eqref{eq:Spin7cyl 1}. 
Similarly, since 
\begin{align}
\left\langle 2 \dot a_t \wedge E_t, i(b^\sharp) * \varphi \right\rangle
&=- 2 * (\dot a_t \wedge E_t \wedge b \wedge \varphi)
= 2 \langle \dot a_t \wedge E_t \wedge \varphi, * b \rangle, \\
- \left\langle E_t^2, b \wedge \varphi \right \rangle 
&=- *(b \wedge \varphi \wedge * E_t^2) 
=- \langle \varphi \wedge * E_t^2, * b \rangle, 
\end{align}
we see that \eqref{eq:Spin7cyl 2 0} is equivalent to \eqref{eq:Spin7cyl 2}. 
\end{proof}

Hence, eliminating $* (\dot a_t \wedge E_t \wedge \varphi)$ from 
\eqref{eq:Spin7cyl 1} by \eqref{eq:Spin7cyl 2}, we obtain 
\begin{align} \label{eq:Spin7cyl 4}
-* \varphi \wedge E_t + \frac{1}{6} E_t^3
+ \frac{1}{2} *(\varphi \wedge * E_t^2) \wedge * E_t 
=
\left( 1-\frac{1}{2} *(\varphi \wedge E_t^2) \right) * \dot a_t. 
\end{align}

\begin{remark} \label{rem:equiv Spin7}
If $ 1-*(\varphi \wedge E_t^2)/2 \neq 0$, 
\eqref{eq:Spin7cyl 1} and \eqref{eq:Spin7cyl 2} are equivalent to \eqref{eq:Spin7cyl 4} 
by Proposition \ref{prop:equiv Spin7 ptwise}. 
\end{remark}

\subsection{The Karigiannis-Leung functional} \label{sec:KL func}
Karigiannis and Leung \cite{karigiannis2009hodge} introduced the functional whose critical points are dDT connections. 
We first review it. 

Let $X^7$ be a compact 7-manifold with a coclosed $G_2$-structure $\varphi$ 
($d * \varphi =0$) 
and let $(L, h) \to X^7$ be a smooth complex Hermitian line bundle. 
Denote by $\cA_0$ the space of Hermitian connections of $(L,h)$. 
Define a 1-form $\Theta$ on $\cA_0$ by 
$$
\Theta_\n (\i b) = \int_X \i b \wedge \left( \frac{1}{6} F_\nabla^3 + F_\nabla \wedge * \varphi  \right) 
$$
for $\n \in \cA_0$ and $\i b \in \i \Omega^1 = T_\n \cA_0$. Then we see that 
$\Theta_\n =0$ if and only if $\n$ is a dDT connection. 
We can show that $\Theta$ is closed as in the proof of Lemma \ref{lem:3-form closed}. 
Since $\cA_0$ is contractible, there exists $\cF: \cA_0 \to \rl$ such that $d \cF = \Theta$. 
Hence we see that dDT connections are critical points of $\cF$. \\

Now, we study the relation between ${\rm Spin}(7)$-dDT connections on $\mathbb{R} \times X^7$ and 
the Karigiannis-Leung functional $\cF$. 
Set 
$$
\cA_{ac}:=\left\{ \n \in \cA_0 \; \middle| \; 
1+\frac{1}{2} *(\varphi \wedge F_\n^2) > 0 \right\}. 
$$
This type of the subset is also considered in the dHYM case. 
For example, see the survey article \cite[Definition 2.1]{collins2021survey}. 
By the mirror of the associator equality in \cite[Theorem 5.1]{kawai2021mirror}, 
it will be natural to call a Hermitian connection $\n$ satisfying 
$1 + *(\varphi \wedge F_\n^2)/2 > 0$ {\bf almost calibrated} as in the dHYM case. 

Define a metric $\cG$ on $\cA_{ac}$ by 
$$
\cG_\n (\sqrt{-1} a, \sqrt{-1} b) = \int_X \langle a, b \rangle_\n 
\left( 1+\frac{1}{2} *(\varphi \wedge F_\n^2) \right) \vol
$$
where $\n \in \cA_{ac}, \sqrt{-1} a, \sqrt{-1} b \in \i \Omega^1 = T_\n \cA_{ac}$, 
$\vol$ is the induced volume form from $\varphi$, 
and $\langle \cdot, \cdot \rangle_\n$ is the induced metric 
on the space of differential forms 
from $(\id_{TX} + (- \i F_\n)^\sharp)^* \varphi$. 
Here, $(- \i F_\n)^\sharp$ is an endomorphism of $TX$ 
defined by 
$g((- \i F_\n)^\sharp (u),v) = - \i F_\n (u,v)$ for $u,v \in TX$, 
where $g$ is the induced metric (on $TX$) from $\varphi$. 
Note that 
$(- \i F_\n)^\sharp$ is skew-symmetric with respect to $g$. 
Explicitly, if we denote by $g_\n$ the induced metric (on $TX$) from 
$(\id_{TX} + (- \i F_\n)^\sharp)^* \varphi$, 
we have $g_\n=(\id_{TX} + (- \i F_\n)^\sharp)^* g$ 
and $\langle \cdot, \cdot \rangle_\n$ is the induced metric from $g_\n$. 

The following is the main theorem of this paper. 

\begin{theorem} \label{thm:gradient flow}
The gradient flow equation of $\cF$ with respect to $\cG$ on $\cA_{ac}$ 
agrees with the ${\rm Spin}(7)$-dDT equation on $\mathbb{R} \times X^7$. 
\end{theorem}

\begin{proof}
We first deduce the gradient flow equation and compare it with the computation 
in Section \ref{sec:Spin7 cyl}. 
Take any $\n \in \cA_{ac}$ and $b \in \Omega^1$. 
Set 
$$
E_\n=-\i F_\n \in \Omega^2. 
$$ 
Denote by $\langle \cdot, \cdot \rangle$ the induced metric 
on the space of differential forms from $\varphi$. 
Then we compute  
\begin{align} \label{eq:gradient flow 1}
(d \cF)_\n  (\i b) 
= \int_X \i b \wedge \left( \frac{1}{6} F_\nabla^3 + F_\nabla \wedge * \varphi  \right)
= \int_X \left \langle b, *\left( \frac{1}{6} E_\nabla^3 - E_\nabla \wedge * \varphi \right) \right \rangle \vol.
\end{align}
By Proposition \ref{prop:pb xi}, we have 
$$
*\left( \frac{1}{6} E_\nabla^3 - E_\nabla \wedge * \varphi \right)
=
\left(\left(\id_{TX} - (E_\n^\sharp)^2 \right)^{-1} \right)^* \eta_\n, 
$$
where 
\begin{align}
\eta_\n
=
* \left( -* \varphi \wedge E_\n + \frac{1}{6} E_\n^3 + \frac{1}{2} *(\varphi \wedge * E_\n^2) \wedge * E_\n \right) 
\in \Omega^1. 
\end{align}
Since 
$$
\id_{TX} - (E_\n^\sharp)^2 
= (\id_{TX} - E_\n^\sharp) (\id_{TX} + E_\n^\sharp) 
= {}^t (\id_{TX} + E_\n^\sharp) (\id_{TX} + E_\n^\sharp), 
$$
where ${}^t (\id_{TX} + E_\n^\sharp)$ is the transpose of $\id_{TX} + E_\n^\sharp$ with respect to $g$, 
we have 
\begin{equation} \label{eq:gradient flow 2}
\begin{split}
\left \langle b, *\left( \frac{1}{6} E_\nabla^3 - E_\nabla \wedge * \varphi \right) \right \rangle
=&
\left \langle b, \left(\left(\id_{TX} - (E_\n^\sharp)^2 \right)^{-1} \right)^* \eta_\n \right \rangle \\
=&
\left \langle \left( \left(\id_{TX} + E_\n^\sharp \right)^{-1} \right)^* b, 
\left( \left(\id_{TX} + E_\n^\sharp \right)^{-1} \right)^* \eta_\n \right \rangle
=
\left \langle b, \eta_\n \right \rangle_\n. 
\end{split}
\end{equation}
Then by \eqref{eq:gradient flow 1} and \eqref{eq:gradient flow 2}, 
the gradient vector field of $\cF$ with respect to $\cG$ is given by 
$$
\cA_{ac} \ni \n \mapsto 
\frac{\i \eta_\n} {1-\frac{1}{2} *(\varphi \wedge E_\n^2)} \in \i \Omega^1. 
$$
Thus a family $\{ \n_t \}_{t \in \mathbb{R}} \subset \cA_{ac}$ 
satisfies the gradient flow of $\cF$ with respect to $\cG$ if and only if 
\begin{align} \label{eq:gradient flow 3}
\dot a_t 
= 
\frac{\eta_{\n_t}} {1-\frac{1}{2} *(\varphi \wedge E_{\n_t}^2)}
=
\frac{* \left( -* \varphi \wedge E_t + \frac{1}{6} E_t^3 + \frac{1}{2} *(\varphi \wedge * E_t^2) \wedge * E_t \right)}{1-\frac{1}{2} *(\varphi \wedge E_t^2)},  
\end{align}
where 
$\n_t=\n_0+\i a_t \cdot \id_L$, $a_t \in \Omega^1$, $\dot a_t = \partial a_t/ \partial t$ 
and 
$E_t=E_{\n_t}=-\i F_{\n_t}$. 
Then we see that \eqref{eq:gradient flow 3} is equivalent to \eqref{eq:Spin7cyl 4}. 
By Remark \ref{rem:equiv Spin7}, this is equivalent to the ${\rm Spin}(7)$-dDT equation on $\mathbb{R} \times X^7$. 
\end{proof}

By Theorem \ref{thm:gradient flow}, we will have to consider the deformation theory 
of the ${\rm Spin}(7)$-dDT connections on $\mathbb{R} \times X^7$ next 
for the analogue of instanton Floer homology for 3-manifolds. 
Deformations of ${\rm Spin}(7)$-dDT connections on a compact manifold 
with a ${\rm Spin}(7)$-structure are studied in \cite[Theorem 1.2]{kawai2021deformationSpin(7)}, 
but there are some technical assumptions. 
We will have to deal with more technical issues, including these, 
to develop the deformation theory on a cylinder.

\appendix
\section{Algebraic Computations}
In this appendix, we give some algebraic computations needed in the proof of 
Theorem \ref{thm:gradient flow}. 

Set $V =\mathbb{R}^7$ and let $g$ be the standard inner product on $V$. 
Denote by $*$ the standard Hodge star operator on $V$. 
For a 2-form $F \in \Lambda^2 V^*$, define $F^\sharp \in {\rm End} (V)$ by 
\[
g(F^\sharp (u), v) = F(u, v)
\]
for $u,v \in V$. 
Then, 
$F^\sharp$ is skew-symmetric, and hence, 
$\det(I + F^\sharp) >0$, where $I$ is the identity matrix. 
We also have 
$$
\det (I - (F^\sharp)^2) 
= \det (I+F^\sharp) \det (I-F^\sharp)
= \det (I+F^\sharp) \det (I+{}^t F^\sharp)
= (\det (I+F^\sharp))^2
>0,  
$$
where we ${}^t F^\sharp$ is the transpose of $F^\sharp$ with respect to $g$.  
Define a 3-form $\varphi \in \Lambda^3 V^*$ by 
$$
\varphi = e^{123} + e^{145} + e^{167} + e^{246} - e^{257} - e^{347} - e^{356}, 
$$
where $\{ e_i \}_{i=1}^7$ 
is a standard oriented basis of $V$ with the dual basis $\{ e^i \}_{i=1}^7$ of $V^\ast$ 
and $e^{i_1 \dots i_k}$ is short for $e^{i_1} \wedge \cdots \wedge e^{i_k}$. 
The stabilizer of $\varphi$ is known to be the exceptional $14$-dimensional simple Lie group 
$G_2$. The elements of $G_2$ preserve the standard inner product $g$ and volume form $\vol$. 
The group $G_2$ acts canonically on $\Lambda^k V^*$, and 
$\Lambda^2 V^*$ is decomposed as 
$\Lambda^2 V^* = \Lambda^2_7 V^* \oplus \Lambda^2_{14} V^*$, 
where $\Lambda^2_\ell V^*$ is a $\ell$-dimensional irreducible subrepresentation of $G_2$ 
in $\Lambda^2 V^*$. 
For more details, see for example \cite[Section 2.2]{kawai2021mirror}. 
Set 
$$
F=F_7 + F_{14} =i(u)\varphi + F_{14} \in \Lambda^2_7 V^* \oplus \Lambda^2_{14} V^*
$$
for $u \in V$.

\begin{proposition}\label{prop:pb xi}
For a 2-form $F \in \Lambda^2 V^*$, 
set $\xi= -* \varphi \wedge F + F^3/6  \in \Lambda^6 V^*$. Then we have 
\begin{align}
(I - (F^\sharp)^2)^* * \xi
=
* \left( \xi + \frac{1}{2} *(\varphi \wedge * F^2) \wedge * F \right). 
\end{align}
\end{proposition}

\begin{proof}
Since $(I - (F^\sharp)^2)^* * \xi = * \xi - ((F^\sharp)^2)^* * \xi$, 
we only have to compute $((F^\sharp)^2)^* * \xi$. 
Set 
$$F_{i j} =F(e_i, e_j).$$ 
We have $F^\sharp=\sum_{i,j}F_{i j} e^i \otimes e_j$, which implies that 
$(F^\sharp)^2 = \sum_{i, j, k} F_{i j} F_{j k} e^i \otimes e_k$. 
Then we compute 
\begin{align}
((F^\sharp)^2)^* * \xi 
= \sum_{i, j, k} F_{i j} F_{j k} * \xi (e_k) \cdot e^i 
= -\sum_j \langle i(e_j)F, * \xi \rangle \cdot i(e_j) F. 
\end{align}
Since 
\begin{align}
\langle i(e_j)F, * \xi \rangle
=&
* \left( * \xi \wedge * (i(e_j)F) \right)
=
-* \left( * \xi \wedge e^j \wedge * F \right)
=
\langle e^j, *(* \xi \wedge * F) \rangle, \\
i(e_j) F
=& -* (e^j \wedge * F), 
\end{align}
we have 
\begin{align} \label{eq:pb xi 1}
((F^\sharp)^2)^* * \xi 
= * \left( * \left( * \xi  \wedge * F \right) \wedge *F \right)
= * \left( * \left( * \left( -*\varphi \wedge F + \frac{F^3}{6} \right)  \wedge * F \right) \wedge *F \right). 
\end{align}
\begin{lemma} \label{lem:pb xi 1}
We have 
\begin{align}
(*F^3) \wedge *F &=0, \\
*(\varphi \wedge * F^2) &= -6 i(u) F. 
\end{align}
\end{lemma}

\begin{proof}
We can prove the first equation as in \cite[Lemma C.2]{kawai2020deformation}. 
For any $v \in V$, set 
$$
v^\flat = g(v, \cdot) \in V^*.
$$
We compute 
\begin{align}
v^\flat \wedge (*F^3) \wedge *F 
= (*F^3) \wedge * (i(v) F)
=F^3 \wedge i(v) F
= i(v)(F^4/4)=0, 
\end{align}
which implies the first equation. 
Similarly, for any $v \in V$, we have 
$$
v^\flat \wedge \varphi \wedge * F^2
=* (v^\flat \wedge \varphi) \wedge F^2
=-i(v) * \varphi \wedge F^2
=* \varphi \wedge i(v) F^2
= 2 i(v) F \wedge F \wedge * \varphi. 
$$
Since 
$F \wedge * \varphi=i(u) \varphi \wedge * \varphi = 3 * u^\flat$ 
by for example \cite[Lemma B.1]{kawai2020deformation}, we obtain 
$$
v^\flat \wedge \varphi \wedge * F^2
=
6 \langle u^\flat, i(v) F \rangle \vol 
=
6 \langle v^\flat \wedge u^\flat, F \rangle \vol 
=
-6 \langle v^\flat, i(u) F \rangle \vol,  
$$
which implies the second equation. 
\end{proof}

Then by \eqref{eq:pb xi 1}, Lemma \ref{lem:pb xi 1} 
and the equation 
$F \wedge * \varphi=i(u) \varphi \wedge * \varphi = 3 * u^\flat$, we obtain 
\begin{align}
((F^\sharp)^2)^* * \xi 
=& * \left( * \left( * \left( -*\varphi \wedge F \right)  \wedge * F \right) \wedge *F \right)\\
=& * \left( * \left( -3 u^\flat  \wedge * F \right) \wedge *F \right)\\
=& 3 * \left( (i(u) F) \wedge * F \right)
= - \frac{1}{2} * \left( *(\varphi \wedge * F^2) \wedge * F \right) 
\end{align}
and the proof is completed. 
\end{proof}

\begin{proposition} \label{prop:equiv Spin7 ptwise}
For a 1-form $a \in V^*$ and a 2-form $F \in \Lambda^2 V^*$ 
such that $1-*(\varphi \wedge F^2)/2 \neq 0$, 
\begin{align}
-* \varphi \wedge F + \frac{1}{6} F^3
- \left( 1-\frac{1}{2} *(\varphi \wedge F^2) \right) * a 
+ * (a \wedge F \wedge \varphi) \wedge * F &=0, \label{eq:appSpin7cyl 1}\\
\frac{1}{2} \varphi \wedge * F^2 - a \wedge F \wedge \varphi &=0 \label{eq:appSpin7cyl 2}
\end{align}
if and only if 
\begin{align} \label{eq:appSpin7cyl 4}
-* \varphi \wedge F + \frac{1}{6} F^3
+ \frac{1}{2} *(\varphi \wedge * F^2) \wedge * F 
=
\left( 1-\frac{1}{2} *(\varphi \wedge F^2) \right) * a. 
\end{align}
\end{proposition}

\begin{proof}
Eliminating $a \wedge F \wedge \varphi$ from \eqref{eq:appSpin7cyl 1} by \eqref{eq:appSpin7cyl 2}, 
we obtain \eqref{eq:appSpin7cyl 4}. 
Conversely, \eqref{eq:appSpin7cyl 4} implies \eqref{eq:appSpin7cyl 2} by the following Lemma \ref{lem:lem gfl}. 
By \eqref{eq:appSpin7cyl 4}, the left hand side of \eqref{eq:appSpin7cyl 1} is computed as 
\begin{align}
- \frac{1}{2} *(\varphi \wedge * F^2) \wedge * F + * (a \wedge F \wedge \varphi) \wedge * F
=
* \left(- \frac{1}{2} \varphi \wedge * F^2 + a \wedge F \wedge \varphi \right) \wedge * F,  
\end{align}
which vanishes by \eqref{eq:appSpin7cyl 2}.
\end{proof}

\begin{lemma} \label{lem:lem gfl}
For any 2-form $F \in \Lambda^2 V^*$, we have 
$$
* \left( -* \varphi \wedge F + \frac{1}{6} F^3
+ \frac{1}{2} *(\varphi \wedge * F^2) \wedge * F \right) \wedge F \wedge \varphi
=
\frac{1}{2} \left(1-\frac{1}{2} *(\varphi \wedge F^2) \right) \varphi \wedge * F^2. 
$$
\end{lemma}

\begin{proof}
Fix any $v \in V$ and set 
$$
J_1= v^\flat \wedge * \left( -* \varphi \wedge F + \frac{1}{6} F^3 \right) \wedge F \wedge \varphi, \qquad
J_2= v^\flat \wedge * \left( \frac{1}{2} *(\varphi \wedge * F^2) \wedge * F \right) \wedge F \wedge \varphi. 
$$
We compute $J_1$ and $J_2$. We have 
\begin{align}
J_1=&
*\left( i(v) \left( * \varphi \wedge F - \frac{1}{6} F^3 \right) \right) \wedge *(2F_7 -F_{14})\\
=&
i(v) \left( * \varphi \wedge F - \frac{1}{6} F^3 \right) \wedge (2F_7 -F_{14})
=
\left( -3 * (v^\flat \wedge u^\flat) - \frac{1}{2} i(v)F \wedge F^2 \right) \wedge (2F_7 -F_{14}),  
\end{align}
where we use $* \varphi \wedge F=3*u^\flat$. 
We also have 
$$
-3 * (v^\flat \wedge u^\flat) \wedge (2F_7 -F_{14})
=
-3 \langle v^\flat \wedge u^\flat, 2F_7 -F_{14} \rangle \vol 
= -3 \langle v^\flat, i(u)F \rangle \vol 
$$
as $i(u) F_7= i(u) i(u) \varphi=0$, and 
\begin{align}
&\left(- \frac{1}{2} i(v)F \wedge F^2 \right) \wedge (2F_7 -F_{14}) \\
=&
- \frac{1}{2} i(v)F \wedge (F_7^2+2F_7 \wedge F_{14} +F_{14}^2) \wedge (2F_7 -F_{14})\\
=&
- \frac{1}{2} \left(i(v)F_7 + i(v)F_{14} \right) \wedge (2 F_7^3+3 F_7^2 \wedge F_{14} - F_{14}^3)\\
=&
- \frac{1}{2} \left \{ i(v)F_7 \wedge (3 F_7^2 \wedge F_{14} - F_{14}^3) 
+ i(v)F_{14} \wedge (2 F_7^3+3 F_7^2 \wedge F_{14})
\right \}, 
\end{align}
where we use $i(v)F_7 \wedge F_7^3=i(v)(F_7^4/4)=0$ and $i(v)F_{14} \wedge F_{14}^3=i(v)(F_{14}^4/4)=0$. 
By \cite[(B.7)]{kawai2020deformation}, we have 
\begin{align} \label{eq:equiv Spin7 ptwise 1}
F_7^3=6|u|^2 * u^\flat. 
\end{align}
Then 
\begin{align}
3 i(v)F_7 \wedge F_7^2 \wedge F_{14}
=& i(v) F_7^3 \wedge F_{14}
=-6|u|^2 *(v^\flat \wedge u^\flat) \wedge F_{14}
=6|u|^2 \langle v^\flat, i(u) F \rangle \vol, \\
2 i(v)F_{14} \wedge F_7^3
=&
12|u|^2 i(v)F_{14} \wedge * u^\flat 
=
12|u|^2 \langle F_{14}, v^\flat \wedge u^\flat \rangle \vol 
=
- 12|u|^2 \langle v^\flat, i(u) F \rangle \vol.  
\end{align}
Hence we obtain 
\begin{align} \label{eq:equiv Spin7 ptwise 2}
J_1=
(-3+3|u|^2) \langle v^\flat, i(u) F \rangle \vol
+
\frac{1}{2} \left( i(v)F_7 \wedge F_{14}^3 -3 i(v)F_{14} \wedge F_7^2 \wedge F_{14} \right). \\
\end{align}

Next, we compute $J_2$. By Lemma \ref{lem:pb xi 1}, we have 
\begin{align}
J_2
=&
v^\flat \wedge * \left(-3 i(u)F \wedge * F \right) \wedge *(2F_7-F_{14}) \\
=&
3 * \left( i(u)F \wedge * F \right) \wedge *\left(i(v) (-2F_7+F_{14}) \right) 
=
3 i(u)F \wedge *F \wedge i(v) (-2F_7+F_{14}). 
\end{align}
Since 
\begin{align}
i(u)F \wedge *F
=& i(u)F_{14} \wedge \left( \frac{1}{2} F_7 \wedge \varphi - F_{14} \wedge \varphi \right)\\
=&
\frac{1}{2} \left( i(u) (F_{14} \wedge F_7 \wedge \varphi) -F_{14} \wedge F_7 \wedge i(u) \varphi \right)
- \frac{1}{2} i(u) F_{14}^2 \wedge \varphi \\
=&
- \frac{1}{2} F_7^2 \wedge F_{14} 
- \frac{1}{2} \left( i(u) (F_{14}^2 \wedge \varphi)-F_{14}^2 \wedge i(u) \varphi \right)\\
=&
\frac{1}{2} \left( |F_{14}|^2 * u^\flat - F_7^2 \wedge F_{14} + F_7 \wedge F_{14}^2 \right),  
\end{align}
we have 
\begin{align}
J_2=
\frac{3}{2} \left(- F_7^2 \wedge F_{14} + F_7 \wedge F_{14}^2 \right) \wedge i(v) (-2F_7+F_{14})
+
\frac{3}{2} |F_{14}|^2 * u^\flat \wedge i(v) (-2F_7+F_{14}). 
\end{align}
We compute 
\begin{align}
&\left(- F_7^2 \wedge F_{14} + F_7 \wedge F_{14}^2 \right) \wedge i(v) (-2F_7+F_{14}) \\
=&
2 i(v) F_7 \wedge F_7^2 \wedge F_{14} - 2 i(v) F_7 \wedge F_7 \wedge F_{14}^2 
- i(v) F_{14} \wedge F_7^2 \wedge F_{14} + i(v) F_{14} \wedge F_7 \wedge F_{14}^2. 
\end{align}
By \eqref{eq:equiv Spin7 ptwise 1}, it follows that 
$$
2 i(v) F_7 \wedge F_7^2 \wedge F_{14}
= \frac{2}{3} i(v) F_7^3 \wedge F_{14} = 4|u|^2 \langle v^\flat, i(u)F \rangle \vol. 
$$
Since 
$- 2 i(v) F_7 \wedge F_7 \wedge F_{14}^2
= -i(v) F_7^2 \wedge F_{14}^2 
= F_7^2 \wedge i(v) F_{14}^2 
= 2 i(v) F_{14} \wedge F_7^2 \wedge F_{14}$, 
we have 
$$
- 2 i(v) F_7 \wedge F_7 \wedge F_{14}^2 - i(v) F_{14} \wedge F_7^2 \wedge F_{14}
= i(v) F_{14} \wedge F_7^2 \wedge F_{14}. 
$$
We also have 
$$
i(v) F_{14} \wedge F_7 \wedge F_{14}^2 
= \frac{1}{3} i(v) F_{14}^3 \wedge F_7 
= - \frac{1}{3} F_{14}^3 \wedge i(v) F_7
$$
and 
$$
\frac{3}{2} |F_{14}|^2 * u^\flat \wedge i(v) (-2F_7+F_{14})
=
\frac{3}{2} |F_{14}|^2 \langle -2F_7+F_{14}, v^\flat \wedge u^\flat \rangle \vol
=
-\frac{3}{2} |F_{14}|^2 \langle v^\flat, i(u)F \rangle \vol. 
$$
Hence we obtain 
\begin{align} \label{eq:equiv Spin7 ptwise 3}
J_2= \left(6|u|^2-\frac{3}{2} |F_{14}|^2 \right) \langle v^\flat, i(u)F \rangle \vol
+\frac{3}{2} i(v) F_{14} \wedge F_7^2 \wedge F_{14}
-\frac{1}{2} i(v) F_{7} \wedge F_{14}^3. \\
\end{align}

Then by \eqref{eq:equiv Spin7 ptwise 2} and \eqref{eq:equiv Spin7 ptwise 3}, we obtain 
\begin{align}
J_1+J_2
= 3 \left(-1+3|u|^2-\frac{1}{2} |F_{14}|^2 \right) \langle v^\flat, i(u)F \rangle \vol 
= 3 \left(-1+\frac{1}{2} *(\varphi \wedge F^2) \right) \langle v^\flat, i(u)F \rangle \vol, 
\end{align}
where we use 
$*(\varphi \wedge F^2)
=* \left(F \wedge *(2F_7-F_{14}) \right) 
= 2|F_7|^2 - |F_{14}|^2 = 6|u|^2-|F_{14}|^2$ 
by \cite[Lemma B.1]{kawai2020deformation}. 
Then it follows that 
$$
* \left( -* \varphi \wedge F + \frac{1}{6} F^3
+ \frac{1}{2} *(\varphi \wedge * F^2) \wedge * F \right) \wedge F \wedge \varphi
=
3 \left(-1+\frac{1}{2} *(\varphi \wedge F^2) \right) *(i(u)F). 
$$
Since $\varphi \wedge * F^2 = -6 * (i(u)F)$ by Lemma \ref{lem:pb xi 1}, 
the proof is completed. 
\end{proof}


\bibliography{references}

\end{document}